\documentclass[11pt]{article}
\usepackage{amsmath,wasysym}
\usepackage{amsthm}
\usepackage{amsrefs}
\usepackage{amssymb,amsfonts}

\usepackage{hyperref}
\usepackage{latexsym}
\usepackage{ytableau}
\usepackage{todonotes}

\usepackage{nicefrac}
\usepackage{epsfig}
\usepackage{stmaryrd}

\usepackage{setspace}
\usepackage{tikz-cd}

\usepackage{enumerate}
\usepackage[all]{xypic}
\usepackage{bbm,ifpdf,tikz,mathtools}
\usepackage{tikz-3dplot}

\usetikzlibrary{decorations.pathreplacing}

\ifpdf
\usepackage{pdfsync}
\fi
\usetikzlibrary{positioning,matrix,arrows,decorations.pathmorphing}
\usepackage{enumitem}

\usetikzlibrary{positioning,arrows}
\usetikzlibrary{decorations.markings}

\usepackage{caption, subcaption}
\usetikzlibrary{arrows.meta,
                bbox}

\newtheorem{theorem}{Theorem}[section]

\newtheorem{lemma}[theorem]{Lemma}

\newtheorem{corollary}[theorem]{Corollary}

\newtheorem{conjecture}[theorem]{Conjecture}
\newtheorem*{main-result}{Main-Result}
{

\theoremstyle{definition}
\newtheorem{definition}[theorem]{Definition}
\newtheorem{example}[theorem]{Example}

\newtheorem{remark}[theorem]{Remark}

}

\newcommand{\excise}[1]{}

\newcommand{\CC}{\mathbb{C}}

\newcommand{\QQ}{\mathbb{Q}}

\newcommand{\RR}{\mathbb{R}}

\newcommand{\frakS}{\mathfrak{S}}

\renewcommand{\and}{\qquad\text{and}\qquad}

\newcommand{\FIS}{\operatorname{FI^\sharp}}

\newcommand{\Conf}{\operatorname{Conf}}

\newcommand{\FBmod}{\operatorname{FB-mod}}
\newcommand{\FImod}{\operatorname{FI-mod}}
\newcommand{\FISmod}{\operatorname{FI^\sharp-mod}}
\newcommand{\FB}{\operatorname{FB}}
\newcommand{\FI}{\operatorname{FI}}
\newcommand{\spn}{\operatorname{span}}

\usepackage{todonotes}

\begin{document}

\spacing{1.2}
\noindent{\Large\bf Equivariant log concavity and the \boldmath{$\FIS$}-module structure on \\ $H^i(\Conf(n,\RR^d))$
}\\

\noindent{\bf Benjamin Homan}\footnote{Supported by NSF grant DMS-2039316.}\\
Department of Mathematics, University of Oregon, Eugene, OR\\
\vspace{.1in}

{\small
\begin{quote}
\noindent {\em Abstract.} Previous work has conjectured that the graded $\frakS_n$-representations $H^\bullet(\Conf(n,\RR^d);\QQ)$ are strongly equivariantly log concave, and has proven this conjecture in low degrees.  By leveraging the theory of representation stability, we are able instead prove a stronger statement about the $\FIS$-module structure on $H^i(\Conf(n,\RR^d);\QQ)$ which implies the original conjecture up to degree 19.  We conjecture that this equivariant log concavity-like property holds in all degrees for the $\FIS$-modules $H^i(\Conf(n,\RR^d);\QQ)$.

\end{quote} }

\section{Introduction}

Given a finite group $G$, we say that a graded representation $V^\bullet$ of $G$ is \textbf{strongly equivariantly log concave} if, for all $i<j\leq k<\ell$ with $i+\ell = j+k$, there exists a $G$-equivariant inclusion:
\begin{equation}\label{SELCCondition}
    V^i\otimes V^\ell \hookrightarrow V^j\otimes V^k.
\end{equation}
Previous work has conjectured that particular representations of the symmetric group $\frakS_n$ are strongly equivariantly log concave \cite{KPY17}*{Conjecture 5.3} \cite{MMPR23}*{Conjecture 1.6}. They are the cohomology rings:
\begin{itemize}
    \item $A_n^\bullet:= H^\bullet(\Conf(n, \CC);\QQ)$,
    \item $C_n^\bullet := H^{2\bullet}(\Conf(n, \RR^3);\QQ)$,
\end{itemize}
where $\Conf(n,X)$ is the space of ordered $n$-tuples of distinct points in a topological space $X$.\footnote{One could consider $H^*(\Conf(n,\RR^d);\QQ)$ for any $d$, however when $d$ is even this is isomorphic to $A_n^\bullet$ up to a change in grading.  Likewise, when $d$ is odd we have an isomorphism with $C_n^\bullet$.}  We say that a graded representation is \textbf{strongly equivariantly log concave in degree \boldmath{$m$}} if (\ref{SELCCondition}) holds for all $i<j\leq k<\ell$ with $i+\ell = j+k =m$. Both $A_n$ and $C_n$ are strongly equivariantly log concave in degree $m$ for all $m\leq 14$ \cite{MMPR23}*{Theorem 1.7}. 

In addition to the structure of graded $\frakS_n$-representation, we have natural maps $A^\bullet_n\to A_{n+1}^\bullet$ and $C^\bullet_n \to C^\bullet_{n+1}$ induced by maps $\Conf(n+1,\RR^d) \to \Conf(n,\RR^d)$ that remove points from a configuration. In the language of representation stability, this gives both $A^\bullet$ and $C^\bullet$ the structure of an FI-module. Previous work leveraged this extra FI-module structure when proving that the graded $\frakS_n$-representations $A^\bullet_n$ and $C_n^\bullet$ are strongly equivariantly log concave in low degrees \cite{MMPR23}*{Theorem 1.7}.

In fact, $A^\bullet$ and $C^\bullet$ carry additional structure arising from maps $A^\bullet_n\to A^\bullet_{n-1}$ and $C^\bullet_n\to C^\bullet_{n-1}$.  These maps are not well-defined on the level of spaces, as there is no canonical way to define a map $\Conf(n-1,\RR^d)\to \Conf(n,\RR^d)$ that adds a point to a given configuration.  However, one can loosely imagine these maps as ``adding a point near infinity,'' and they give $A^\bullet$ and $C^\bullet$ a much stronger structure of an $\FIS$-module.  Ultimately, we find that strong equivariant log concavity not only appears when considering $A^\bullet$ and $C^\bullet$ as graded representations, but a similar property also holds at the level of $\FIS$-modules. We conjecture the following:

\begin{conjecture}\label{mainConjecture}
    Given positive integers with $i<j\leq k <\ell$ and $i+\ell = j+k$, there is an inclusion of $\FIS$-modules:
    \[A^i\otimes A^\ell \hookrightarrow A^j\otimes A^k,\]
    and
    \[C^{i}\otimes C^{\ell}\hookrightarrow C^{j}\otimes C^{k}.\]
\end{conjecture}

Further, we prove this conjecture holds in low degree.

\begin{theorem}\label{maintheorem}
    Given positive integers $i<j\leq k<\ell$ with $i+\ell =j+k=m$, for $m\leq 19$ we have an inclusion of $\FIS$-modules
    \[A^i\otimes A^\ell \hookrightarrow A^j\otimes A^k,\]
    and
    \[C^{i}\otimes C^{\ell}\hookrightarrow C^{j}\otimes C^{k}.\]
\end{theorem}

\begin{corollary}\label{maincorollary}
    For all $n$, the graded representations $A_n^\bullet$ and $C_n^{\bullet}$ are strongly equivariantly log concave up to degree 19.
\end{corollary}

While \cite{MMPR23}*{Theorem 1.7} proves that these inclusions are $\frakS_n$-equivariant, we show that they also respect the rigid $\FIS$-module structures on $A^i$ and $C^i$. As we will see in Remark \ref{stableBound}, the stronger $\FIS$-module structure can be completely determined with fewer computations.  Hence, we are able to prove Conjecture \ref{mainConjecture} up to degree 19 and obtain strong equivariant log concavity of the graded representations as a corollary.

First, we review the necessary results of representation theory that make these calculations possible. Then, by using the equivalence of the categories $\FBmod$ and $\FISmod$, we observe that the determining the $\FIS$-module structures of $A^i\otimes A^j$ and $C^i\otimes C^j$ requires fewer computations than explicitly calculating them as $\FB$-modules.  Finally, we describe the algorithm that calculates the $\FIS$-module structures of $A^i\otimes A^j$ and $C^i\otimes C^j$ and proves Theorem \ref{maintheorem}. We run this algorithm using the software package SageMath \cite{Sage}.

\vspace{\baselineskip}
\noindent
{\em Acknowledgments:}
We are grateful to Nicholas Proudfoot for his mentorship, guidance, and comments throughout the preparation of this manuscript. We would also like to thank Eric Ramos for the insights into determining the structure of a tensor product of $\FIS$-modules. Finally, we thank Galen Dorpalen-Barry and the Texas A\&M compute cluster Whistler for providing the computing power necessary to prove Theorem \ref{maintheorem}.

\section{Representation Stability}

In this section we review the definitions and properties of $\FB$, $\FI$, and $\FIS$-modules. For a more thorough treatment see \cite{CEF15}.  Given a partition $\lambda$ of a positive integer $n$, we denote by $V_\lambda$ the corresponding irreducible $\mathfrak{S}_n$-representation.  For the sake of this paper, all representations will be over the field $\mathbb{Q}$. 

\subsection{Stabilization of \boldmath$\FB$-Modules}

For any fixed integer $i\geq 0$, we can view $A^i$ and $C^i$ as a sequence of $\frakS_n$-representations $(A^i_n)$ and $(C^i_n)$ by varying $n$. A categorical description of this data is an \boldmath{$\FB$}\unboldmath\textbf{-module}:

\begin{definition}\label{defFBmod}
    Let $\FB$ be the category whose objects are finite sets whose morphisms are bijections between the sets.  An \boldmath{$\FB$}\textbf{-module}\unboldmath ~is any functor from $\FB$ to $k$-Mod for some commutative ring $k$.\footnote{For this paper we will only consider $\QQ$-modules.}  Given an $\FB$-module $V$, we will denote $V_n := V[n]$ where $[n]=\{1,2,\dots n\}$.
\end{definition}

\begin{example}\label{FBmod}
    Given a sequence $\{W_n\}$ of $\mathfrak{S}_n$-representations, we can construct an $\FB$-module $W$ by declaring $W[n] = W_n$.  Then, for any bijection $\sigma:[n]\to [n]$, $\sigma$ acting on $W_n$ defines the map $W(\sigma):W_n\to W_n$.  In particular, this gives $A^i$ and $C^i$ the structure of an $\FB$-module for any fixed integer $i\geq 0$, as $A^i_n$ and $C^i_n$ have an $\frakS_n$-action induced by permuting the points in a configuration.
\end{example}

Some $\FB$-modules exhibit an interesting stabilization phenomenon.  As an example, consider $A^1=H^1(\Conf(-,\RR^2),\mathbb{Q})$:

\begin{example}\label{stabilityEx}
   For this example, let the Young diagram of a partition represent the corresponding irreducible $\frakS_n$-representation.
    \begin{center}
        \begin{tabular}{c||c}
            $n$ & $A^1$ \\
            \hline
            & \\
            $1$ &  $0$ \\
            & \\
            $2$ & $\ydiagram{2}$ \\
            & \\
            $3$ & $\ydiagram{3} \bigoplus  \ydiagram{2,1}$\\
            & \\
            $4$ & $\ydiagram{4} \bigoplus \ydiagram{3,1} \bigoplus \ydiagram{2,2}$ \\
            & \\
            $5$ & $\ydiagram{5} \bigoplus \ydiagram{4,1} \bigoplus \ydiagram{3,2}$ \\
            & \\
            $6$ & $\ydiagram{6} \bigoplus \ydiagram{5,1} \bigoplus \ydiagram{4,2}$ \\
        \end{tabular}
    \end{center}

    The table continues for all $n$, but notice that for $n>4$, the tableaux indexing the decomposition of $A^1_n$ come from adding a box to the first row of those in the decomposition of $A^1_{n-1}$. We call this phenomenon \textbf{stabilization}. 
    
\end{example}

\begin{definition}\label{stabilizationDef}
    Suppose $V$ is an $\FB$-module such that the irreducible composition of $V_{m}$ is given by:
    \[V_m = \bigoplus_{\lambda\vdash m} (V_\lambda)^{\oplus c_\lambda}.\]
    We say that $V$ \textbf{stabilizes at} \boldmath{$m $}\unboldmath  ~if for all $n>m$:
    \[V_n = \bigoplus_{\lambda\vdash m}(V_{(\lambda_1+(n-m),\lambda_2,\dots,\lambda_\ell)})^{\oplus c_\lambda}.\]
\end{definition}

Example \ref{stabilityEx} is no coincidence; for all $i\geq 0$, $A^i$ and $C^i$ will stabilize \cite{C2012}*{Theorem 1}, and this stabilization was key to the proof of \cite{MMPR23}*{Theorem 1.7}.  To understand this behavior, we must add additional structure to our $\FB$-modules.

\begin{definition}\label{defFImod}
    Let $\FI$ be the category of finite sets with injections. An \textbf{FI-module} is a functor $\text{FI}\to k\text{-mod}$ for some commutative ring $k$.
\end{definition}

\begin{example}\label{FreeFIModDef}
    For any $m\geq 0$, we define an $\FI$-module \boldmath$M(m)$ \unboldmath that assigns to any finite set $S$ the $\QQ$-module with basis given by injections $[m]\hookrightarrow S$.  Then, any injection of finite sets $\varphi:S\to T$ induces a map $\varphi_*:M(m)(S)\to M(m)(T)$. When $m=0$, there is only one injection $\emptyset \hookrightarrow S$. Thus, $M(0)$ is a constant $\FI$-module assigning to each finite set the trivial representation of $\frakS_n$.  Additionally, $M(1)_n$ has a basis indexed by elements of $[n]$, since injections $\{1\} \hookrightarrow [n]$ are completely determined by the image of $1$.  The $\mathfrak{S_n}$-action on $M(1)_n$ is given by permuting the index set $[n]$, and thus $M(1)_n$ is the permutation representation of $\frakS_n$.
\end{example}

\begin{example}\label{exFImod}
     For any fixed $i$, $A^i$ and $C^i$ have the structure of an $\FI$-module.  Given an injection $f:[m]\to [n]$, we define a map $\Conf(n,\RR^d) \to \Conf(m,\RR^d)$ by $(p_1,\dots, p_n) \mapsto (p_{f(1)},\dots,p_{f(m)})$.  The induced maps on cohomology are $\QQ$-module homomorphisms $A^i(f): A^i_m\to A^i_n$ and $C^i(f): C^i_m\to C^i_n$.  When $f: [n]\to[n+1]$ this corresponds to removing a point from the configuration as discussed in the introduction.
\end{example}

\begin{remark}\label{FIhasFBStructure}
    There is a forgetful functor is induced by the inclusion $\FB\hookrightarrow\FI$.  We denote this functor $\pi:\FImod\to \FBmod$ and it forgets all non-bijective morphisms.
\end{remark}

The additional structure of an $\FI$-module allows for a characterization of the $\FB$-modules which stabilize. 

\begin{definition}{\cite{CEF15}*{Proposition 2.3.5}}\label{finGenFI}
    An $\FI$-module $V$ is \textbf{finitely generated}\footnote{This is not the definition used in \cite{CEF15}*{Definition 2.3.4}. However, \cite{CEF15}*{Proposition 2.3.5} gives that this characterization is equivalent.} if there exists a finite sequence of integers $\{m_i\}$ and a surjection:
    \[\bigoplus_{i} M(m_i) \twoheadrightarrow V,\]
    where $M(m_i)$ is the $\FI$-module defined in Example \ref{FreeFIModDef}.
\end{definition}

\begin{theorem}{\cite{CEF15}*{Theorem 1.13}}\label{finGenEqualsStable}
    If $V$ is an $\FB$-module over a field of characteristic $0$, then $V$ stabilizes if and only if there exists a finitely generated $\FI$-module $W$ such that $V=\pi(W)$.
\end{theorem}

Importantly, both $A^i$ and $C^i$ are finitely generated $\FI$-modules and thus the $\FB$-modules $\pi(A^i)$ and $\pi(C^i)$ stabilize \cite{C2012}*{Theorem 1}.  In fact, there are sharp bounds for this stabilization:

\begin{theorem}{\cite{HR17}*{Theorem 1.1}}\label{sharpStabilityBound} For all integers $i\geq 0$, the $\FB$-modules $A^i$ and $C^i$ stabilize sharply at $n=3i+1$ and $n=3i$, respectively.
\end{theorem}

The formulation of a strong equivariant log concavity property akin to (\ref{SELCCondition}) requires a notion of the tensor product of $\FB$-modules:

\begin{definition}\label{tensorFIFBdef}
    Given two $\FB$-modules $V$ and $W$, their tensor product $V\otimes W$ is the $\FB$-module assigning:
    \[(V\otimes W)_n = V_n \otimes W_n.\]
    If $V$ and $W$ are $\FI$-modules, then $V\otimes W$ also has an $\FI$-module structure; an injection $f: [m]\to [n]$ defines a map $V\otimes W(f):=V(f)\otimes W(f):V_m\otimes W_m \to V_n\otimes W_n$.
\end{definition}

If two $\FB$-modules stabilize, it follows that their tensor product will stabilize:

\begin{theorem}{\cite{MMPR23}*{Theorem 3.3}}\label{tensorStabilizationDegree}
    If $\FB$-modules $V$ and $W$ stabilize at $n$ and $m$, respectively, then $V\otimes W$ stabilizes at $n+m$.
\end{theorem}

\begin{corollary}\label{AiandCiStabilizationDegree}
    For any integer $i\geq 0$, the $\FB$-modules $A^i\otimes A^j$ and $C^i\otimes C^j$ stabilize at $3(i+j)+2$ and $3(i+j)$, respectively.
\end{corollary}

\begin{proof}
    Theorem \ref{sharpStabilityBound} and Theorem \ref{tensorStabilizationDegree} immediately imply the corollary.
\end{proof}

Corollary \ref{AiandCiStabilizationDegree} gives a finite bound on the computations necessary to prove equivariant log concavity in low degrees.  We need only compute the irreducible decompositions of $(A^i\otimes A^j)_n$ and $(C^i\otimes C^j)_n$ up to the stabilization bounds $n=3(i+j)+2$ and $n=3(i+j)$ and verify (\ref{SELCCondition}) for these representations.

\subsection{FI\boldmath{$^\sharp$}-Modules}

Determining the structure of $A^i\otimes A^j$ and $C^i\otimes C^j$ by computing their underlying $\FB$-modules up to stabilization is a computation intensive process that requires determining Kronecker coefficients for large $\frakS_n$-representations.  We improve this approach by observing that both $A^i$ and $C^i$ have a stronger \boldmath$\FIS$\unboldmath\textbf{-module} structure:

\begin{definition}{\cite{CEF15}*{Definition 4.1.1}}\label{FISharpDef}
    Let $\FIS$ be the category of finite sets with \textbf{partially-defined injections}.  A partially-defined injection between finite sets $S$ and $T$ are subsets $A\subseteq S$ and $B\subseteq T$ together with a bijection $\phi:A\to B$.  An \boldmath{$\FIS$}\unboldmath\textbf{-module} is a functor $\FIS\to k$-mod for a commutative ring $k$. 
\end{definition}

\begin{remark}
    Let co-$\FI$ be the category $\FI^{\text{op}}$. One can think of $\FIS$-modules as having both an $\FI$-module and a co-$\FI$-module structure that are compatible with each other. An $\FIS$-module structure is very rigid, but we will see in Remark \ref{stableBound} that it is exactly this rigidity that allows us to completely determine the $\FIS$-module structure of $A^i\otimes A^j$ and $C^i\otimes C^j$ more efficiently.
\end{remark}

\begin{example}
    Both $A^i$ and $C^i$ are $\FIS$-modules because $\Conf(n,\RR^d)$ is a \textbf{homotopy FI\boldmath{$^\sharp$}\unboldmath-space} \cite{CEF15}*{Proposition 6.4.2}.  That is, it defines a functor FI$^\sharp\to $ hTop, the category of topological spaces with homotopy classes of maps as morphisms.  As discussed in Example \ref{exFImod}, the $\FI$-module structure on cohomology is induced by deleting points from a configuration. Additionally, while  ``adding a point at infinity" is not well defined at the level of spaces, it is up to homotopy. This induces a co-$\FI$ structure on cohomology.
\end{example}

\begin{remark}\label{ForgetfulFIStoFI}
    Just like in Remark \ref{FIhasFBStructure} there is a map $\FI\hookrightarrow\FIS$ taking only the morphisms where the bijective component is defined on the whole domain.  This induces another forgetful functor $\FIS$-mod $\to\FI$-mod.
\end{remark}

Ultimately, the categories $\FISmod$ and $\FBmod$ are equivalent \cite{CEF15}*{Theorem 4.1.5}.  To define this equivalence, we first consider a functor $M: \FBmod \to \FImod$.  Recall from Remark \ref{FIhasFBStructure}, that $\pi: \FImod\to \FBmod $ is the forgetful functor induced by the inclusion of $\FB\hookrightarrow \FI$.

\begin{definition}{\cite{CEF15}*{Definition 2.2.2}}
    The functor $M:\FBmod\to \FImod$ is the left-adjoint of the map $\pi:\FImod \to \FBmod$.  Explicitly, given an $\FB$-module $W$, the $\frakS_n$-representations $M(W)_n$ are given by:
    \begin{equation*}\label{defM}
        \displaystyle M(W)_n=\bigoplus_{a\leq n} \operatorname{Ind}_{\frakS_a\times \frakS_{n-a}}^{\frakS_n} W_a \boxtimes \mathbb{Q},
    \end{equation*}
    where $\frakS_{n-a}$ acts on $\QQ$ trivially.
\end{definition}

\begin{example}
    In Example \ref{FreeFIModDef} we defined the FI-module $M(m)$.  This FI-module is the image of the regular representation $\QQ[\frakS_m]$ under the functor $M$.
\end{example}

\begin{remark}\label{A1Generator}
    By using \cite{HR17}*{Corollary 2.10}, one can confirm that in Example \ref{stabilityEx} $A^1\cong M(V_{(2)})$ where we consider the representation $V_{(2)}$ as the following $\FB$-module:
    \[\left(V_{(2)}\right)_n = \begin{cases}V_{(2)} \quad\quad \text{if } n=2; \\
    0 \quad\quad\quad \text{else.}\end{cases}\]
    In fact, \cite{HR17}*{Corollary 2.10} allows us to compute the $\FB$-modules that map via $M$ to $A^i$ and $C^i$ for any $i$.
\end{remark}

It happens that $\FI$-modules in the image of $M$ always have a stronger $\FIS$-module structure \cite{CEF15}*{Example 4.1.4}, so we can take $M$ to be a functor $\FBmod \to \FIS\text{-mod}$.  An explicit description of the other half of the equivalence requires that we define the \textbf{span} of elements in an FI-module:

\begin{definition}{\cite{CEF15}*{Definition 2.3.1}}
    Let $V$ be an $\FI$-module and $A$ be some collection of elements in $\bigsqcup_{n}V_n$.  We say $\spn_V(A)$ is the minimal sub-$\FI$-module of $V$ that has every element of $A$.
\end{definition}

We now define a functor $H_0:\FImod \to\FBmod$. For any positive integer $n$ define $V_{<n}$ to be the $\FB$-module that contains only the data of the $\mathfrak{S}_i$-modules $V_i$ for $i<n$.  That is:

\begin{equation*}
    (V_{<n})_m = \begin{cases}
        V_m \quad\quad\quad~\text{if } m<n; \\
        0 \quad\quad\quad\quad \text{else.}
    \end{cases}
\end{equation*}

\begin{definition}{\cite{CEF15}*{Definition 2.3.7}} The functor $H_0:\FImod\to\FBmod$ is the left-adjoint of inclusion $\FBmod\hookrightarrow \FImod$ where any non-bijective maps are sent to the zero morphism of $\QQ$-modules.  Explicitly, given an $\FI$-module $V$ and finite set $S$:
\begin{equation}\label{H0expliciteq}
    H_0(V)_S=V_S/\spn(V_{<|S|})_S
\end{equation}
\end{definition}

In Remark \ref{ForgetfulFIStoFI} we mentioned a forgetful functor $\FISmod\to \FImod$.  If we restrict $H_0$ to $\FISmod$ by precomposing this forgetful functor, $H_0$ together with $M$ define an equivalence of categories:

\begin{theorem}\cite{CEF15}*{Theorem 4.1.5}\label{FBFISEquivalence}
    The category of $\FIS$-modules is equivalent to the category of $\FB$-modules via the functors $M: \FBmod \to \FISmod$ and $H_0:\FISmod\to \FBmod$.
\end{theorem}

\begin{corollary}
    For every $\text{FI}^\sharp$-module $V$ there is an $\FB$-module $W$ such that $V=M(W)$.  In particular, $W=H_0(V)$.
\end{corollary}

The main result of this paper requires checking (\ref{SELCCondition}) at the level of $\FIS$-modules.  Leveraging the equivalence of $\FIS$-mod and $\FB$-mod, we instead calculate the $\FB$-modules $W_{i,\ell}$ and $W_{j,k}$ so that $M(W_{i,\ell})=A^i\otimes A^\ell$ and $M(W_{j,k})=A^j\otimes A^k$. Then, we can check the containment $W_{i,\ell}\hookrightarrow W_{j,k}$ of $\FB$-modules to imply containment of $\FIS$-modules. In other words, we need to calculate $H_0(A^i\otimes A^\ell)$ and $H_0(A^j\otimes A^k)$ and the corresponding $\FB$-modules for $C^\bullet$ as well.

To this end, we need an explicit expression for $H_0(V)_i$ for any $\FIS$-module $V$ and positive integer $i$.  When $V$ is an $\FIS$-module we can use Theorem \ref{FBFISEquivalence} to simplify (\ref{H0expliciteq}). We have a recursive definition for $H_0(V)$ given by:

\begin{equation}
    H_0(V)_n=V_n/M(H_0(V)_{<n})_n
\end{equation}

Before performing an explicit calculation, we need to confirm that $H_0(V)_n$ is eventually zero.  Otherwise, the process of recursively calculating $H_0(V)$ will not terminate.

\begin{lemma}\label{generatingBound}
    Let $j$ and $k$ be positive integers and $m>2(k+j)$.  Then $H_0(A^j\otimes A^k)_m=0$ and $H_0(C^{j} \otimes C^{k})_m=0$.
\end{lemma}

\begin{proof}
    By \cite{HR17}*{Corollary 2.10} we have that $H_0(A^j)$ and $H_0(C^{j})$ are non-zero in degrees $j+1$ to $2j$. Suppose $V$ and $W$ are finitely-generated $\FI$-modules such that $H_0(V)$ and $H_0(W)$ are zero in all degrees strictly greater than $n$ and $m$, respectively.  Then \cite{CEF15}*{Proposition 2.3.6, Remark 2.3.8}  gives that $H_0(V\otimes W)$ is zero in degrees strictly greater than $n+m$. Hence, both $H_0(A^j\otimes A^k)$ and $H_0(C^{j}\otimes C^{k})$ are zero in degrees strictly greater that $2(i+j)$ as desired.
\end{proof}

\begin{remark}\label{stableBound}
    This bound is lower than the stabilization degree for these $\FB$-modules, which are $3(j+k)+2$ for $A^j\otimes A^k$ and $3(j+k)$ for $C^j\otimes C^k$ by Corollary \ref{AiandCiStabilizationDegree}.
\end{remark}

\begin{example}\label{H0Calculation}
    We will compute $H_0(A^1\otimes A^1)$.  We noted in Remark \ref{A1Generator} that $A^1 =M(V_{(2)})$ so $(A^1\otimes A^1)_i$ will first be non-zero in degree 2. In general, if the minimal degree of a generator for $V$ is $n$ and for $W$ the minimal degree is $m$, we know that $H_0(V\otimes W)_i=0$ for all $i<\max(n,m)$.  We also know from Lemma \ref{generatingBound} that $H_0(A^1\otimes A^1)_i=0$ for any $i>4$ so we have a finite number of computations that we can perform in SageMath \cite{Sage}:
    \begin{center}
    \begin{tabular}{c|c|c|c}
        $i$ & 2 & 3 & 4  \\
        \hline\hline
        &&& \\
        $(A^1\otimes A^1)_i$ & $V_{(2)}$ & $V_{(1,1,1)}\oplus V_{(2,1)}^{\oplus 3}\oplus V_{(3)}^{\oplus 2}$ & $V_{(1,1,1,1)}\oplus V_{(2,1,1)}^{\oplus 3}\oplus V_{(2,2)}^{\oplus 4}\oplus V_{(3,1)}^{\oplus 5}\oplus V_{(4)}^{\oplus 3}$ \\
        
        \hline
        &&& \\
        $M(H_0(A^1\otimes A^1)_{<i})_i$ & $0$ & $V_{(2,1)}\oplus V_{(3)}$ & $V_{(1,1,1,1)}\oplus V_{(2,1,1)}^{\oplus 3} \oplus V_{(2,2)}^{\oplus 3} \oplus V_{(3,1)}^{\oplus 4} \oplus V_{(4)}^{\oplus 2}$\\
        \hline 
        &&& \\
        $H_0(A^1\otimes A^1)_i$ & $V_{(2)}$ & $V_{(2,1)}^{\oplus 2} \oplus V_{(3)}$ & $V_{(2,2)}\oplus V_{(3,1)}\oplus V_{(4)}$ \\
        \hline
    \end{tabular}
    \end{center}
    Hence, $A^1\otimes A^1 = M\left(V_{(2)}\oplus V_{(2,1)}^{\oplus2}\oplus V_{(3)}\oplus V_{(2,2)}\oplus V_{(3,1)}\oplus V_{(4)}\right)$.
\end{example}

\begin{example}\label{H0ComputationEx}
    These computations become complex quickly.  For example, the first calculations required to prove our theorem are $H_0(A^1\otimes A^3)$ and $H_0(A^2\otimes A^2)$.  Both are zero in degrees strictly greater than 8 and less than 3 so we need only consider the following representations:
    \begin{center}
        \begin{tabular}{c|p{2.8in}|p{2.8in}}
            $n$ & $H_0(A^1\otimes A^3)_n$ & $H_0(A^2\otimes A^2)_n$ \\
            \hline\hline
            && \\
             3 & $0$ & $V_{(1,1,1)}\oplus V_{(2,1)} \oplus V_{(3)}$ \\
             && \\
             \hline
             && \\
             4 & $V_{(1,1,1,1)}\oplus V_{(2,1,1)}^{\oplus 5} \oplus V_{(2,2)}^{\oplus 2} \oplus V_{(3,1)}^{\oplus 5} \oplus V_{(4)}$ & $V_{(1,1,1,1)}^{\oplus 4}\oplus V_{(2,1,1)}^{\oplus 13} \oplus V_{(2,2)}^{\oplus 9} \oplus V_{(3,1)}^{\oplus 13} \oplus V_{(4)}^{\oplus 5}$ \\
             && \\
             \hline 
             && \\
             5 & $V_{(1,1,1,1,1)}^{\oplus 2} \oplus V_{(2,1,1,1)}^{\oplus 9} \oplus V_{(2,2,1)}^{\oplus 13} \oplus V_{(3,1,1)}^{\oplus 16} \oplus V_{(3,2)}^{\oplus 14} \oplus V_{(4,1)}^{\oplus 12} \oplus V_{(5)}^{\oplus 3}$ & $V_{(1,1,1,1,1)}^{\oplus 4} \oplus V_{(2,1,1,1)}^{\oplus 19} \oplus V_{(2,2,1)}^{\oplus 26} \oplus V_{(3,1,1)}^{\oplus 33} \oplus V_{(3,2)}^{\oplus 29} \oplus V_{(4,1)}^{\oplus 25} \oplus V_{(5)}^{\oplus 7}$  \\
             && \\
             \hline
             && \\
             6 & $V_{(2, 1, 1, 1, 1)}^{\oplus 2} \oplus _{(2, 2, 1, 1)}^{\oplus 9} \oplus V_{(2, 2, 2)}^{\oplus 4} \oplus V_{(3, 1, 1, 1)}^{\oplus 10} \oplus V_{(3, 2, 1)}^{\oplus 21} \oplus V_{(3, 3)}^{\oplus 9} \oplus V_{(4, 1, 1)}^{\oplus 16} \oplus V_{(4, 2)}^{\oplus 13} \oplus V_{(5, 1)}^{\oplus 9} \oplus V_{(6)}$ & $V_{(2, 1, 1, 1, 1)}^{\oplus 5} \oplus _{(2, 2, 1, 1)}^{\oplus 14} \oplus V_{(2, 2, 2)}^{\oplus 10} \oplus V_{(3, 1, 1, 1)}^{\oplus 19} \oplus V_{(3, 2, 1)}^{\oplus 36} \oplus V_{(3, 3)}^{\oplus 13} \oplus V_{(4, 1, 1)}^{\oplus 26} \oplus V_{(4, 2)}^{\oplus 26} \oplus V_{(5, 1)}^{\oplus 16} \oplus V_{(6)}^{\oplus 6}$ \\
             && \\

             \hline 

             && \\
             7 & $V_{(2,2,2,1)}\oplus V_{(3,1,1,1,1)}\oplus V_{(3,2,1,1)}^{\oplus5}\oplus V_{(3,2,2)}^{\oplus 4}\oplus V_{(3,3,1)}^{\oplus 5}\oplus V_{(4,1,1,1)}^{\oplus 4} \oplus V_{(4,2,1)}^{\oplus 10}\oplus V_{(4,3)}^{\oplus 5}\oplus V_{(5,1,1)}^{\oplus 5}\oplus V_{(5,2)}^{\oplus 5}\oplus V_{(6,1)}^{\oplus 2}$ & $V_{(2,2,1,1,1)}\oplus V_{(2,2,2,1)}^{\oplus 2}\oplus V_{(3,1,1,1,1)}\oplus V_{(3,2,1,1)}^{\oplus8}\oplus V_{(3,2,2)}^{\oplus 7}\oplus V_{(3,3,1)}^{\oplus 8}\oplus V_{(4,1,1,1)}^{\oplus 6} \oplus V_{(4,2,1)}^{\oplus 15}\oplus V_{(4,3)}^{\oplus 8}\oplus V_{(5,1,1)}^{\oplus 8}\oplus V_{(5,2)}^{\oplus 9}\oplus V_{(6,1)}^{\oplus 4}\oplus V_{(7)}$ \\

             && \\
             \hline
             && \\
             8 & $V_{(3,3,2)}\oplus V_{(4,2,1,1)}\oplus V_{(4,3,1)}^{\oplus2}\oplus V_{(5,1,1,1)}\oplus V_{(5,2,1)}\oplus V_{(5,3)}\oplus V_{(6,1,1)}$ & $V_{(3,3,1,1)} \oplus V_{(3,3,2)} \oplus V_{(4,2,1,1)} \oplus V_{(4,2,2)}\oplus V_{(4,3,1)}^{\oplus 2} \oplus V_{(4,4)}\oplus V_{(5,1,1,1)}\oplus V_{(5,2,1}^{\oplus 2}\oplus V_{(5,3)}\oplus V_{(6,1,1)}\oplus V_{(6,2)}$\\

             && \\
             \hline
        \end{tabular}
    \end{center}
    We can see that $H_0(A^1\otimes A^3)_n \subseteq H_0(A^2\otimes A^2)_n$ for all $n$ and thus we have containment of $\FB$-modules.  Then the equivalence of categories from Theorem \ref{FBFISEquivalence} gives that $A^1\otimes A^3 \subseteq A^2\otimes A^2$ as $\FIS$-modules, as desired.
\end{example}

\begin{remark}
    If we performed the computation in Example \ref{H0Calculation} using the approach of \cite{MMPR23}*{Theorem 1.7}, we would have to calculate $A^1\otimes A^3$ and $A^2\otimes A^2$ explicitly up to stabilization.  Theorem \ref{tensorStabilizationDegree} says stabilization does not occur until  degree 13.  In contrast, our computation terminates at degree 8.

    Further, by calculating $H_0(A^1\otimes A^3)$ and $H_0(A^2\otimes A^2)$ instead of the $\FB$-modules $A^1\otimes A^3$ and $A^2\otimes A^2$, not only did we perform fewer calculations, but we also conclude that there is a containment of the $\FIS$-modules $A^1\otimes A^3\subseteq A^2\otimes A^2$
\end{remark}

\begin{example}
    The containment of $\text{FI}^\sharp$-modules we prove in Theorem \ref{maintheorem} is indeed stronger than that for $\text{$\FB$}$-modules. Recall that $\pi: \text{FI-mod} \to \text{$\FB$-mod}$ is the forgetful functor induced by the inclusion $\text{$\FB$} \hookrightarrow \text{FI}$. We will exhibit  $\FI$-modules $Y,Z$ such that $\pi(Y)\supseteq \pi(Z)$ as $\text{$\FB$-modules}$, but $Y\not\supseteq Z$ as $\FI$-modules. Thus, Conjecture \ref{mainConjecture} is stronger than \cite{MMPR23}*{Conjecture 1.6}.

    Consider the $\FI$-modules $Y=M(V_{(1)}) \oplus M(V_{(2,1)})$ and $Z= M(V_{(2)})$.  Clearly, $Y\not\supseteq Z$ as $\FI$-modules.  Then, by \cite{HR17}*{Lemma 2.2} we need only confirm that $\pi(Y)_n \supseteq \pi(Z)_n$ for $n\leq 5$, at which point both sequences stabilize.  But we have:
\begin{center}
    \begin{tabular}{c||c|c}
         $n$ & $\pi(Y)_n$ & $\pi(Z)_n$ \\
         \hline
         && \\
         $1$ & $V_{(1)}$ & $0$ \\
        \hline
        && \\
         $2$ & $V_{(1,1)} \oplus V_{(2)}$ & $V_{(2)}$ \\
         \hline
         && \\
         $3$ & $V_{(2,1)}^{\oplus 2} \oplus V_{(3)}$ & $V_{(2,1)} \oplus V_{(3)}$ \\
         \hline
         && \\
         $4$ & $V_{(2,1,1)}\oplus V_{(2,2)}\oplus V_{(3,1)}^{\oplus2}\oplus V_{(4)}$ & $V_{(2,2)}\oplus V_{(3,1)}\oplus V_{(4)}$ \\
         \hline
         && \\
         $5$ & $V_{(2,2,1)}\oplus V_{(3,1,1)} \oplus V_{(3,2)} \oplus V_{(4,1)}^{\oplus2} \oplus V_{(5)}$ & $V_{(3,2)}\oplus V_{(4,1)} \oplus V_{(5)}$\\
         \hline
    \end{tabular}
\end{center}
and so $\pi(Y)\supseteq \pi(Z)$.
\end{example}

\section{Proof}

We are now able to give the computational proof of Theorem \ref{maintheorem}.  The code used to prove this theorem is available at \url{https://github.com/bhoman-math/ELCandFISharp}.

\begin{proof}[Proof of Theorem \ref{maintheorem}:] 
As in Example \ref{H0Calculation} we calculate $H_0(A^j\otimes A^k)$ and $H_0(A^i\otimes A^\ell)$ for all $i<j\leq k<\ell$ when $i+\ell=j+k\leq 19$.  This reduces to a finite number of computations that can be run in SageMath \cite{Sage}.  We then confirm that $H_0(A^j\otimes A^k)\supseteq H_0(A^i\otimes A^\ell)$ as $\FB$-modules.  We perform the same calculations with $H_0(C^j\otimes C^k)$ and $H_0(C^i\otimes C^\ell)$. Then the equivalence of categories given in Theorem \ref{FBFISEquivalence} proves our theorem in these low degrees.
\end{proof}

\bibliographystyle{amsrefs}
\bibliography{refs}

@article {MMPR23,
    AUTHOR = {Matherne, Jacob P. and Miyata, Dane and Proudfoot, Nicholas
              and Ramos, Eric},
     TITLE = {Equivariant log concavity and representation stability},
   JOURNAL = {Int. Math. Res. Not. IMRN},
  FJOURNAL = {International Mathematics Research Notices. IMRN},
      YEAR = {2023},
    NUMBER = {5},
     PAGES = {3885--3906},
      ISSN = {1073-7928,1687-0247},
   MRCLASS = {20F55 (05B35 20C30 52C35 55R80)},
  MRNUMBER = {4565658},
       DOI = {10.1093/imrn/rnab352},
       URL = {https://doi.org/10.1093/imrn/rnab352},
}

@article {HR17,
    AUTHOR = {Hersh, Patricia and Reiner, Victor},
     TITLE = {Representation stability for cohomology of configuration
              spaces in {$\Bbb R^d$}},
      NOTE = {With an appendix written jointly with Steven Sam},
   JOURNAL = {Int. Math. Res. Not. IMRN},
  FJOURNAL = {International Mathematics Research Notices. IMRN},
      YEAR = {2017},
    NUMBER = {5},
     PAGES = {1433--1486},
      ISSN = {1073-7928,1687-0247},
   MRCLASS = {20C15 (55R80)},
  MRNUMBER = {3658170},
MRREVIEWER = {Sadok\ Kallel},
       DOI = {10.1093/imrn/rnw060},
       URL = {https://doi.org/10.1093/imrn/rnw060},
}

@article {CEF15,
    AUTHOR = {Church, Thomas and Ellenberg, Jordan S. and Farb, Benson},
     TITLE = {F{I}-modules and stability for representations of symmetric
              groups},
   JOURNAL = {Duke Math. J.},
  FJOURNAL = {Duke Mathematical Journal},
    VOLUME = {164},
      YEAR = {2015},
    NUMBER = {9},
     PAGES = {1833--1910},
      ISSN = {0012-7094,1547-7398},
   MRCLASS = {20C30 (05E10 18B99 55R80)},
  MRNUMBER = {3357185},
MRREVIEWER = {Nadia\ P.\ Mazza},
       DOI = {10.1215/00127094-3120274},
       URL = {https://doi.org/10.1215/00127094-3120274},
}

@misc{Sage,
    TITLE = {SageMath Mathematical Software System}, 
    URL ={https://www.sagemath.org./}, 
    AUTHOR={The Sage Developers}, 
    YEAR={2021}
}

@article {KPY17,
    AUTHOR = {Gedeon, Katie and Proudfoot, Nicholas and Young, Benjamin},
     TITLE = {The equivariant {K}azhdan-{L}usztig polynomial of a matroid},
   JOURNAL = {J. Combin. Theory Ser. A},
  FJOURNAL = {Journal of Combinatorial Theory. Series A},
    VOLUME = {150},
      YEAR = {2017},
     PAGES = {267--294},
      ISSN = {0097-3165,1096-0899},
   MRCLASS = {05B35 (05E05 05E18 20C30)},
  MRNUMBER = {3645577},
MRREVIEWER = {Eva\ Ferrara Dentice},
       DOI = {10.1016/j.jcta.2017.03.007},
       URL = {https://doi.org/10.1016/j.jcta.2017.03.007},
}

@article {C2012,
    AUTHOR = {Church, Thomas},
     TITLE = {Homological stability for configuration spaces of manifolds},
   JOURNAL = {Invent. Math.},
  FJOURNAL = {Inventiones Mathematicae},
    VOLUME = {188},
      YEAR = {2012},
    NUMBER = {2},
     PAGES = {465--504},
      ISSN = {0020-9910,1432-1297},
   MRCLASS = {55R80},
  MRNUMBER = {2909770},
MRREVIEWER = {Donald\ W.\ Kahn},
       DOI = {10.1007/s00222-011-0353-4},
       URL = {https://doi.org/10.1007/s00222-011-0353-4},
}

\end{document}